\documentclass[10pt,a4paper]{article}
\usepackage[left=2.5cm,right=2.5cm,top=3cm,bottom=3cm]{geometry}
\usepackage{latexsym,amsthm,amssymb,amscd,amsmath,mathtools,stmaryrd,tikz-cd,graphicx}
\usepackage{float}
\usepackage[colorlinks]{hyperref}
\usepackage[thinlines]{easytable}
\usepackage{microtype}
\usepackage{cancel}

\allowdisplaybreaks
\theoremstyle{plain}
\newtheorem{theorem}{Theorem}[section]

\newtheorem{proposition}[theorem]{Proposition}
\newtheorem{corollary}[theorem]{Corollary}
\theoremstyle{definition}
\newtheorem{definition}[theorem]{Definition}
\newtheorem{remark}[theorem]{Remark}
\newtheorem{example}[theorem]{Example}

\input{tcilatex}
\begin{document}

\title{Generalized Lie Algebroids - Examples by Distinguished Lie Algebroids
with Applications to Optimal Control}
\author{Constantin M. ARCUS, Esmaeil PEYGHAN, Esa SHARAHI}
\date{}
\maketitle

\begin{abstract}
We will prove that the generalized Lie algebroid is a distinguished example
by Lie algebroid. The generality of it with respect to the Lie algebroid is
similar with the generality of the pull-back vector bundle with respect to
the vector bundle. Next, we will prove that the proof of Theorem 3.1 from 
\cite{19} is a misconception and the mentioned Theorem has no validity.
Finally, we anatomize an optimal control problem solvable in the generalized
Lie algebroid framework whereas Lie algebroid instrumentation can not solve
it.


\end{abstract}

%
%==================================================
%%%%%%%%%%%%%%%%%%%%%%%%%%%%%

\section{Introduction}

Important applications of Lie algebras in physics and mechanics (see \cite%
{SW}) inspired many authors to study these spaces and to generalize them to
other spaces such as Lie superalgebras \cite{K0}, affine (Kac-Moody) Lie
algebras \cite{K}, quasisimple Lie algebras \cite{HT}, (locally) extended
affine Lie algebras \cite{AABGP, MY, N} and invariant affine reflection
algebras \cite{N1}.

A natural generalization of the usual Lie algebra over the field $\mathbb{K}$
were introduced by Palais \cite{P} and Rinehart \cite{R}. It is called 
\textit{Lie d-ring} or \textit{Lie-Rinehart algebra} (see \cite{H, H1, H2}
for more detailes). For an easy access, see the following definition of
Lie-Rinehart algebra.

\begin{definition}
A Lie-Rinehart algebra over a field $\mathbb{K}$ is a triple $\left( \left(
A,\left[ ,\right] _{A}\right) ,\left( \mathcal{F},\cdot \right) ,\rho
\right) $ such that

$LR1.$ $\left( A,\left[ ,\right] _{A}\right) $ is a Lie algebra over $%
\mathbb{K}$;

$LR2.$ $\left( \mathcal{F},\cdot \right) $ is a commutative and unitary
algebra over $\mathbb{K}$;

$LR3.$ $A$ is a module over $\mathcal{F}$;$\ $

$LR4.$ The modules morphism $\rho $ from $A$ to $Der\left( \mathcal{F}%
\right) $ (see Definition 2.4)$,$ called anchor map, is an algebras morphism
and satisfy the compatibility condition\qquad 
\begin{equation*}
\left[ u,f\cdot v\right] _{A}=f\cdot \left[ u,v\right] _{A}+\rho \left(
u\right) \left( f\right) \cdot v,
\end{equation*}%
for any $u,v\in A$ and $f\in \mathcal{F}$.\ 
\end{definition}

The study of Lie algebroids were considerably improved by J. Pradines in 
\cite{[40]}. He noticed that the Lie algebroids are infinitesimal versions
of Lie groupoids in a functorial manner.

\begin{remark}
Using the general framework of Lie-Rinehart algebras, a Lie algebroid can be
regarded as a triple $((F,\nu ,N),[,]_{F},(\rho ,Id_{N})),$ where $\left(
F,\nu ,N\right) $ is a vector bundle, $\left[ ,\right] _{F}$ is an operation
on the module of sections $\Gamma \left( F,\nu ,N\right) $ and $\left( \rho
,Id_{N}\right) $ is a vector bundles morphism from $\left( F,\nu ,N\right) $
to $\left( TN,\tau _{N},N\right) $ such that the triple 
\begin{equation*}
(\left( \Gamma (F,\nu ,N),[,]_{F}\right) ,\left( \mathcal{F}\left( N\right)
,\cdot \right) ,\Gamma (\rho ,Id_{N})),
\end{equation*}%
is a Lie-Rinehart algebra over $\mathbb{R},$ where $\Gamma (\rho ,Id_{N})$
is the modules morphism associated to the vector bundles morphism $(\rho
,Id_{N})$ (for a detailed illustration, see Definition \ref{Arcus1}).
\end{remark}

%Using an arbitrary Lie algebroid $((F,\nu ,N),[,]_{F},(\rho ,Id_{N}))$, Alan
%Weinstein proposed in \cite{Weinstein} a developement of a Lagrangian
%formalism directly on a Lie algebroid. This is similar to Klein's formalism
%for ordinary Lagrangian Mechanics \cite{Klein}. Alan Weinstein gave a theory
%of Lagrangian systems on Lie algebroids and obtained the Euler-Lagrange
%equations using the dual of a Lie algebroid and the Legendre transformation
%defined by a regular Lagrangian. P. Liberman in \cite{Libermann} showed
%later that such a formalism is not possible if one considers the tangent
%bundle of a Lie algebroid as a space for developing the theory. E. Mart\'{\i}%
%nez in \cite{Martines 1} gave a full description using the prolongation of a
%Lie algebroid presented by K. Mackenzie and P. J. Higgins in \cite{Higins}.
%Also, in the papers \cite{Cortes 1, Cortes 2} and \cite{LP4,LP5} were
%developed the geometry of prolongation Lie algebroid of a Lie algebroid $%
%((F,\nu ,N),[,]_{F},(\rho ,Id_{N})).$

A first generalization of the Lie algebroid is the skew-algebroid. It was
introduced by J. Grabowski and P. Urba\'{n}ski in \cite{[17], [18]} and it
was used in analytical mechanics \cite{[10], [12]}. Also, we remark that the
interests for skew algebroids is given by the natural geometric framework
offered for the nonholonomic mechanics (see \cite{[14]}). For more
applications of skew-algebroids, see e.g. \cite{[7]}.

A possible generalization of the Lie algebroid was introduced recently in
literature by Arcu\c{s} in \cite{2} and is called generalized Lie algebroid
by the following

\begin{definition}
A generalized Lie algebroid is a triple $\left( (F,\nu ,N),[,]_{F,h},\left(
\rho ,\eta \right) \right) $ given by the diagrams 
\begin{equation*}
\begin{array}{c}
\begin{array}[b]{ccccc}
(F,[,]_{F,h}) & ^{\underrightarrow{~\ \ \ \rho \ \ \ \ }} & (TM,[,]_{TM}) & 
^{\underrightarrow{~\ \ \ Th\ \ \ \ }} & (TN,[,]_{TN}) \\ 
~\downarrow \nu &  & ~\ \ \downarrow \tau _{M} &  & ~\ \ \ \tau
_{N}\downarrow \ \ \ \ \ \ \  \\ 
N & ^{\underrightarrow{~\ \ \ \eta ~\ \ }} & M & ^{\underrightarrow{~\ \ \
h~\ \ }} & N%
\end{array}%
\end{array}%
,
\end{equation*}%
where $h$ and $\eta $ are arbitrary diffeomorphisms, $(\rho ,\eta )$ is a
vector bundles morphism from $(F,\nu ,N)$ to $(TM,\tau _{M},M)$ and the
operation 
\begin{equation*}
\begin{array}{ccc}
\Gamma (F,\nu ,N)\times \Gamma (F,\nu ,N) & ^{\underrightarrow{~\ \
[,]_{F,h}~\ \ }} & \Gamma (F,\nu ,N) \\ 
(u,v) & \longmapsto & \ [u,v]_{F,h}%
\end{array}%
,
\end{equation*}%
fulfills 
\begin{equation*}
\lbrack u,f\cdot v]_{F,h}=f\cdot \lbrack u,v]_{F,h}+\left( \Gamma (Th\circ
\rho ,h\circ \eta )(u)\right) \left( f\right) \cdot v,
\end{equation*}%
such that the couple $(\Gamma (F,\nu ,N),[,]_{F,h})$ is a Lie algebra over $%
\mathcal{F}(N).$ The anchor map $\Gamma (Th\circ \rho ,h\circ \eta )$ is
given by the equality%
\begin{equation*}
\left( \Gamma (Th\circ \rho ,h\circ \eta )(u^{\alpha }t_{\alpha })\right)
\left( f\right) =u^{\alpha }\rho _{\alpha }^{i}\cdot \left( \frac{\partial
\left( f\circ h\right) }{\partial x^{i}}\circ h^{-1}\right) ,
\end{equation*}%
for any $u^{\alpha }t_{\alpha }\in \Gamma (F,\nu ,N)$ and $f\in \mathcal{F}%
\left( N\right) $.
\end{definition}

%In \textit{Theorem 4.2} we prove that the conclusion of \textit{Theorem 3.1}
%of the paper \cite{19} is false.
First of all, studying the definition of generalized Lie algebroid, we
remark that the operation $[,]_{F,h}$\ is not $\mathcal{F}(N)$-bilinear and
so, the couple $(\Gamma (F,\nu ,N),[,]_{F,h})$\ can not be regarded as a Lie
algebra over $\mathcal{F}(N)$\ in the usual sense. In \cite{7} a new term,
prealgebra, was introduced in order to extend the notion of algebra.
Basically information about prealgebras are presented in the second part of
Section 2 of this paper. Using the notion of (almost) Lie (pre)algebra, we
obtain a new extension of the usual Lie algebra, called generalized (almost)
Lie algebra.

\begin{definition}
A generalized (almost) Lie algebra over unitary and commutative ring $%
\mathcal{F}$ is a triple $(A,[,]_{A},\left( \rho ,\rho _{0}\right) )$
satisfying

$1.$ $(A,[,]_{A})$ is a (almost) Lie prealgebra over $\mathcal{F}$;

$2.$ $\left( \rho ,\rho _{0}\right) $ is a modules morphism from $\left(
A,+,\cdot \right) $ to $\left( Der(\mathcal{F}),+,\cdot \right) $ (called
anchor map) such that $\rho _{0}$ is invertible and satisfies the
compatibility condition 
\begin{equation*}
\lbrack u,f\cdot v]_{A}=f\cdot \lbrack u,v]_{A}+\rho _{0}^{-1}\left( \rho
(u)(\rho _{0}\left( f\right) )\right) \cdot v,
\end{equation*}%
for any $u,v\in A$ and $f\in \mathcal{F}$.

The generalized (almost) Lie algebra $(A,[,]_{A},\left( \rho ,Id_{\mathcal{F}%
}\right) )$ will be denoted $(A,[,]_{A},\rho ).$
\end{definition}

More information about generalized (almost) Lie algebras is presented in
Section 3 (see also \cite{7,8}). In \textit{Example 3.8 }we prove that the
class of generalized (almost) Lie algebras is larger than the class of
Lie-Rinehart algebras. So, the affirmation \textit{"...... the author's
\textquotedblleft generalized Lie algebra\textquotedblright\ is actually a
particular case of a Lie pseudoalgebra (Lie-Rinehart algebra) mentioned in
the introduction."}\textbf{\ }from \cite{19}, pag. 5 is completely false.

\begin{remark}
A skew-algebroid/Lie algebroid can be regarded as a triple $\left( (F,\nu
,N),[,]_{F},\left( \rho ,Id_{N}\right) \right) $ given by the diagrams 
\begin{equation*}
\begin{array}{c}
\begin{array}[b]{ccc}
(F,[,]_{F}) & ^{\underrightarrow{~\ \ \ \rho \ \ \ \ }} & (TN,[,]_{TN}) \\ 
~\downarrow \nu &  & ~\ \ \ \tau _{N}\downarrow \ \ \ \ \ \ \  \\ 
N & ^{\underrightarrow{~\ \ \ Id_{N}~\ \ }} & N%
\end{array}%
\end{array}%
,
\end{equation*}%
where $(\rho ,Id_{N})$ is a vector bundles morphism from $(F,\nu ,N)$ to $%
(TN,\tau _{N},N)$ and $[,]_{F}$ is an operation on $\Gamma (F,\nu ,N)$ such
that the triple 
\begin{equation*}
\left( \Gamma \left( F,\nu ,N\right) ,\left[ ,\right] _{F},\Gamma (\rho
,Id_{N})\right),
\end{equation*}%
is a generalized almost Lie algebra/generalized Lie algebra over $\mathcal{F}%
(N)$.
\end{remark}

Secondly, studying the definition of generalized Lie algebroid, we remark
that the anchor map $\Gamma \left( Th\circ \rho ,h\circ \eta \right) $ is
only a notation, because we can not discuss about the vector bundles
morphism $\left( Th\circ \rho ,h\circ \eta \right) .$ We can discuss only
about the composition vector bundles morphism $\left( Th,h\right) \circ
\left( \rho ,\eta \right) .$

In the Theorem 3.11, we prove that the generalized Lie algebroid is an
example by distinguished Lie algebroid. The generality of it with respect to
the Lie algebroid is similar with the generality of the pull-back vector
bundle with respect to the vector bundle. Every vector bundle can be
regarded as the pull-back of it through identity. Similar, every Lie
algebroid can be regarded as a particular generalized Lie algebroid such
that $\eta =Id_{N}=h$. This was the motivation for the name of generalized
Lie algebroid.

Section 4, is devoted to show that the proof of Theorem 3.1 (as the only
result) from \cite{19} is a misconception and so why the authors of \cite{19}
have wrong ratiocination.

Using the Euclidean $3$-dimensional manifold $\Sigma $ with the
differentiable structure given by the differentiable atlass $\left\{ \left(
\Sigma ,\varphi _{\Sigma }\right) \right\} ,$ where 
\begin{equation*}
\begin{array}{ccc}
\Sigma & ^{\underrightarrow{~\ \ \ \varphi _{\Sigma }~\ \ }} & \mathbb{R}^{3}
\\ 
x & \longmapsto & \left( x^{1},x^{2},x^{3}\right)%
\end{array}%
,
\end{equation*}%
in Section 5, we put the optimal control problem of finding the curve $\left[
0,T\right] ~^{\underrightarrow{~\ \ \ c~\ \ }}~\Sigma $ given by 
\begin{equation*}
\left( \varphi _{\Sigma }\circ c\right) \left( t\right) =\left( x^{1}\left(
t\right) ,x^{2}\left( t\right) ,x^{3}\left( t\right) \right)
\end{equation*}%
and the sections $u=y^{i}\frac{\partial }{\partial x^{i}}\in \Gamma \left(
T\Sigma _{\mid \text{Im}c},\tau _{\Sigma },\text{Im}c\right) $ which verify
the control system%
\begin{equation}
\begin{array}{c}
\frac{dx^{1}}{dt}=-x^{2}y^{2}+y^{3}, \\ 
\frac{dx^{2}}{dt}=-x^{1}y^{1}-x^{2}y^{2}+y^{3}, \\ 
\frac{dx^{3}}{dt}=y^{1},%
\end{array}
\label{OP}
\end{equation}%
and which are solutions of an ODE by Lagrange type, where $L$ is the
Lagrange fundamental function given by 
\begin{equation*}
L\left( x,y\right) =\frac{1}{2}\left[ \left( y^{1}\right) ^{2}+\left(
y^{2}\right) ^{2}+\left( y^{3}\right) ^{2}\right] .
\end{equation*}

We know that, using the dual of a Lie algebroid and the Legendre
transformation defined by a regular Lagrangian, Alan Weinstein gave a theory
of Lagrangian systems on Lie algebroids and obtained the Euler-Lagrange
equations. Similar to Klein's formalism for ordinary Lagrangian Mechanics 
\cite{Klein}, Alan Weinstein proposed in \cite{Weinstein} a developement of
a Lagrangian formalism directly on a Lie algebroid.

P. Liberman showed in \cite{Libermann} that it is not possible to develope
this formalism, if one considers the tangent bundle of a Lie algebroid as a
space for developing the theory. E. Mart\'{\i}nez in \cite{Martines 1} gave
a full description of a Lagrangian formalism using the prolongation of a Lie
algebroid presented by K. Mackenzie and P. J. Higgins in \cite{Higins}.

In the paper \cite{3}, C. M. Arcu\c{s} presented a Lagrangian formalism
using the commutative diagrams 
\begin{equation*}
\begin{array}{ccccccc}
F & ^{\underrightarrow{~\ \ g~\ \ }} & (F,[,]_{F,h}) & ^{\underrightarrow{~\
\ \rho ~\ \ }} & TN & ^{\underrightarrow{~\ \ Th~\ \ }} & TN \\ 
~\ \downarrow \nu &  & ~\ \downarrow \nu &  & ~\ \ \downarrow \tau _{N} &  & 
~\ \ \downarrow \tau _{N} \\ 
N & ^{\underrightarrow{~\ \ h~\ \ }} & N & ^{\underrightarrow{~\ \ \eta ~\ \ 
}} & N & ^{\underrightarrow{~\ \ h~\ \ }} & N%
\end{array}%
,
\end{equation*}%
where $\left( g,h\right) $ is a locally invertible vector bundles morphism.

After some calculus, in the end of Section 5, we pass the diagrams 
\begin{equation*}
\begin{array}{ccccccccc}
&  & T\Sigma & ^{\underrightarrow{~\ \ g~\ \ }} & (F,[,]_{F,s_{O}}) & ^{%
\underrightarrow{~\ \ \rho ~\ \ }} & T\Sigma & ^{\underrightarrow{~\ \
Ts_{O}~\ \ }} & T\Sigma \\ 
& \dot{c}\nearrow & ~\ \downarrow \tau _{\Sigma } &  & \downarrow \tau
_{\Sigma } &  & ~\ \ \downarrow \tau _{\Sigma } &  & ~\ \ \downarrow \tau
_{\Sigma } \\ 
I & ^{\underrightarrow{~\ \ c~\ \ }} & \Sigma & ^{\underrightarrow{~\ \
s_{O}~\ \ }} & \Sigma & ^{\underrightarrow{~\ \ Id_{\Sigma }~\ \ }} & \Sigma
& ^{\underrightarrow{~\ \ s_{O}~\ \ }} & \Sigma%
\end{array}%
,
\end{equation*}%
where the vector bundle $\left( T\Sigma ,\tau _{\Sigma },\Sigma \right) $ is
anchored by the generalized Lie algebroid $\left( (F,\nu
,N),[,]_{F,s_{O}},\left( \rho ,Id_{\Sigma }\right) \right) $ with the help
of a left invertible vector bundles morphism $\left( g,s_{O}\right) .$

It is important to remark that our optimal control problem can not be solve
with the help of the previous theories of Lagrangian formalism for
(generalized) Lie algebroids, because $\dim \Gamma \left( T\Sigma ,\tau
_{\Sigma },\Sigma \right) =3\neq 2=\dim \Gamma \left( F,\tau _{\Sigma
},\Sigma \right) $. So, the introduction of generalized Lie algebroids is
motivated by the usual problems from optimal control theory

\section{Preliminaries}

\ \ \ \ \ 

In this section, we present basic notions about modules. All the examples
are from the geometry of vector bundles and all the vector bundles that we
used have paracompact basis. Also, we remark that if $\left( A,+\right) $ is
a commutative group, then $(End(A),+,\circ )$ is an unitary ring.\newline

\begin{definition}
If $\left( \mathcal{F},+,\cdot \right) $ is an unitary ring and there exists
an unitary rings morphism 
\begin{equation*}
\begin{array}{ccc}
\mathcal{F} & ^{\underrightarrow{~\ \ \phi ~\ \ }} & End\left( A\right) \\ 
f & \longmapsto & \phi _{f}%
\end{array}%
,
\end{equation*}%
then we will say that $\mathcal{F}$ acts on $A$ with the help of
representations $\phi _{f},~f\in \mathcal{F}$ and the triple $\left(
A,+,\phi \right) $\emph{\ }will be called\emph{\ }module\ over\emph{\ }$%
\mathcal{F}.$\emph{\ }In addition, if the application $\phi $ is injective,
then $\left( A,+,\phi \right) $ will be called faithful module over $%
\mathcal{F}.$ In particular, if $\phi _{f}\left( u\right) \overset{put}{=}%
f\cdot u,$ for any $f\in \mathcal{F}$ and $u\in A,$ then we will say that $%
\left( A,+,\cdot \right) $\emph{\ }or $A$ is a module\ over\emph{\ }$%
\mathcal{F}$.
\end{definition}

\begin{example}
If $\left( F,\nu ,N\right) $ is a vector bundle$,$ then the set of sections $%
\Gamma \left( F,\nu ,N\right) $ can be regarded as a faithful module over $%
\mathcal{F}\left( N\right) $ with respect to the usual action "$\cdot $"$.$
\end{example}

\begin{example}
If $\left( F,\nu ,N\right) $ and $\left( E,\pi ,M\right) $ are two vector
bundles and $\varphi _{0}\in \mathbf{Man}\left( N,M\right) $, then $\left(
\Gamma \left( F,\nu ,N\right) ,+,\odot \right) $ is a module over $\mathcal{F%
}\left( M\right) ,$ where $\odot $ is the action given by 
\begin{equation*}
g\odot z=\varphi _{0}^{\ast }\left( g\right) \cdot z,
\end{equation*}%
for any $g\in \mathcal{F}\left( M\right) $ and $z\in \Gamma \left( F,\nu
,N\right) .$ We denoted by $\mathbf{Man}$ the category of manifolds.
\end{example}

\begin{definition}
If $\left( A,+,\phi \right) $ is a module over unitary ring $\mathcal{F}$
and $\left( B,\boxplus ,\psi \right) $ is a module\emph{\ }over unitary ring 
$\mathcal{G}$, then we define the morphisms set with the source $\left(
A,+,\phi \right) $ and the target $\left( B,\boxplus ,\psi \right) $ by 
\begin{equation*}
\left\{ \left( \alpha ,\alpha _{0}\right) \in Hom\left( A,B\right) \times
Hom\left( \mathcal{F},\mathcal{G}\right) :\alpha \left( \phi _{f}\left(
u\right) \right) =\psi _{\alpha _{0}\left( f\right) }\left( \alpha \left(
u\right) \right) ,~\forall \left( f\in \mathcal{F},~u\in A\right) \right\} .
\end{equation*}%
Sometime, the modules morphism $\left( \alpha ,Id_{\mathcal{F}}\right) $
will be denoted by $\alpha .$ The category of modules will be denoted by $%
\mathbf{Mod}$.
\end{definition}

\begin{definition}
\label{Arcus1} Let $\left( F,\nu ,N\right) $ and $\left( E,\pi ,M\right) $
be two vector bundles. The pair $\left( \varphi ,\varphi _{0}\right) $ given
by the commutative diagram 
\begin{equation*}
\begin{array}{ccc}
~\ \ F & ^{\underrightarrow{~\ \ \varphi ~\ \ }} & ~\ \ E \\ 
\nu \downarrow & ~\  & \pi \downarrow \\ 
~\ \ N & ^{\underrightarrow{~\ \ \varphi _{0}~\ \ }} & ~\ \ M \\ 
~\ \left( t_{\alpha }\right) &  & ~\ \left( s_{a}\right)%
\end{array}%
,
\end{equation*}%
is a vector bundles morphism from $\left( F,\nu ,N\right) $ to $\left( E,\pi
,M\right) $, if there exists $\Phi _{\alpha }^{a}\in \mathcal{F}\left(
M\right) $ such that the application $\Gamma \left( \varphi ,\varphi
_{0}\right) $ from $\left( \Gamma \left( F,\nu ,N\right) ,+,\odot \right) $
over $\mathcal{F}\left( M\right) $ to $\left( \Gamma \left( E,\pi ,M\right)
,+,\cdot \right) $ over $\mathcal{F}\left( M\right) $ given by%
\begin{equation*}
\Gamma \left( \varphi ,\varphi _{0}\right) \left( z^{\alpha }\odot t_{\alpha
}\right) =z^{\alpha }\cdot \Gamma \left( \varphi ,\varphi _{0}\right) \left(
t_{\alpha }\right) =\left( z^{\alpha }\cdot \Phi _{\alpha }^{a}\right) \cdot
s_{a},
\end{equation*}%
is a modules morphism. In literature, $\varphi $ is a vector bundles
morphism covering $\varphi _{0}$. The modules morphism $\Gamma \left(
\varphi ,\varphi _{0}\right) $ will be called the modules morphism
associated to the vector bundles morphism $\left( \varphi ,\varphi
_{0}\right) .$
\end{definition}

Note that $\varphi _{\mid F_{x}}$ is a real vector spaces morphism from $%
F_{x}$ to $E_{\varphi _{0}\left( x\right) }$ satisfying 
\begin{equation*}
\varphi \left( \left( z^{\alpha }\odot t_{\alpha }\right) \left( x\right)
\right) =\left( \Gamma \left( \varphi ,\varphi _{0}\right) \left( z^{\alpha
}\odot t_{\alpha }\right) \right) \left( \varphi _{0}\left( x\right) \right)
.
\end{equation*}
The following is a locally representation of a vector bundle morphism.

\begin{proposition}
Let $\left( F,\nu ,N\right) $ and $\left( E,\pi ,M\right) $ be two vector
bundles. The pair $\left( \varphi ,\varphi _{0}\right) $ given by the
commutative diagram 
\begin{equation*}
\begin{array}{ccc}
~\ \ F & ^{\underrightarrow{~\ \ \varphi ~\ \ }} & ~\ \ E \\ 
\nu \downarrow & ~\  & \pi \downarrow \\ 
~\ \ N & ^{\underrightarrow{~\ \ \varphi _{0}~\ \ }} & ~\ \ M \\ 
~\ \left( t_{\alpha }\right) &  & ~\ \left( s_{a}\right)%
\end{array}%
,
\end{equation*}%
where $\varphi _{0}\in Diff\left( N,M\right) ,$ is a vector bundles morphism
if and only if there exists $\varphi _{\alpha }^{a}\in \mathcal{F}\left(
N\right) $ such that 
\begin{equation*}
\Gamma \left( \varphi ,\varphi _{0}\right) \left( z^{\alpha }\cdot t_{\alpha
}\right) =\left( z^{\alpha }\cdot \varphi _{\alpha }^{a}\right) \circ
\varphi _{0}^{-1}\cdot s_{a}.
\end{equation*}
\end{proposition}

\begin{proof}
Because $\left( \varphi ,\varphi _{0}\right) $ is a vector bundles morphism
from $\left( F,\nu ,N\right) $ to $\left( E,\pi ,M\right) $ if and only if
\begin{eqnarray*}
\Gamma \left( \varphi ,\varphi _{0}\right) \left( z^{\alpha }\cdot t_{\alpha
}\right) &=&\Gamma \left( \varphi ,\varphi _{0}\right) \left( \left(
z^{\alpha }\circ \varphi _{0}^{-1}\right) \odot t_{\alpha }\right) =\left[
\left( z^{\alpha }\circ \varphi _{0}^{-1}\right) \cdot \Phi _{\alpha }^{a}%
\right] \cdot s_{a} \\
&=&\left[ z^{\alpha }\cdot \Phi _{\alpha }^{a}\circ \varphi _{0}\right]
\circ \varphi _{0}^{-1}\cdot s_{a}\overset{\varphi _{\alpha }^{a}=\Phi
_{\alpha }^{a}\circ \varphi _{0}}{=}\left( z^{\alpha }\cdot \varphi _{\alpha
}^{a}\right) \circ \varphi _{0}^{-1}\cdot s_{a},
\end{eqnarray*}%
we obtain the conclusion of the proposition.
\end{proof}

\begin{corollary}
Let $\left( F,\nu ,N\right) $ and $\left( E,\pi ,N\right) $ be two vector
bundles. The pair $\left( \varphi ,Id_{N}\right) $ given by the commutative
diagram 
\begin{equation*}
\begin{array}{ccc}
~\ \ F & ^{\underrightarrow{~\ \ \varphi ~\ \ }} & ~\ \ E \\ 
\nu \downarrow & ~\  & \pi \downarrow \\ 
~\ \ N & ^{\underrightarrow{~\ \ Id_{N}~\ \ }} & ~\ \ N \\ 
~\ \left( t_{\alpha }\right) &  & ~\ \left( s_{a}\right)%
\end{array}%
\end{equation*}%
is a vector bundles morphism if and only if there exists $\varphi _{\alpha
}^{a}\in \mathcal{F}\left( N\right) $ such that%
\begin{equation*}
\Gamma \left( \varphi ,Id_{N}\right) \left( z^{\alpha }\cdot t_{\alpha
}\right) =\left( z^{\alpha }\cdot \varphi _{\alpha }^{a}\right) \cdot s_{a}.
\end{equation*}
\end{corollary}

\begin{definition}
A prealgebra over $\mathcal{F}$ is a pair $\left( A,\left[ ,\right]
_{A}\right) ,$ where $A$ is a module over $\mathcal{F}$ and the operation $%
\left[ ,\right] _{A}$ is biadditive. In particular, if the operation $%
[,]_{A} $ is bihomogenouse, then the pair $\left( A,\left[ ,\right]
_{A}\right) $ is called algebra over $\mathcal{F}$.
\end{definition}

\begin{example}
The set $Der(\mathcal{F})$ of derivations of the unitary ring $\mathcal{F}$
is a module over $\mathcal{F}$ and the operation $[,]_{Der(\mathcal{F}%
)}^{\circ }$ given by 
\begin{equation*}
\lbrack X,Y]_{Der(\mathcal{F})}^{\circ }(f)=X\circ Y\left( f\right) -Y\circ
X\left( f\right) ,~\forall f\in \mathcal{F},
\end{equation*}%
satisfy the compatibility condition 
\begin{equation*}
\lbrack X,f\cdot Y]_{Der(\mathcal{F})}^{\circ }=f\cdot \lbrack X,Y]_{Der(%
\mathcal{F})}^{\circ }+X(f)\cdot Y,
\end{equation*}%
for any $X,Y\in Der(\mathcal{F})$ and $f\in \mathcal{F}.$ So, $\left( Der(%
\mathcal{F}),[,]_{Der(\mathcal{F})}^{\circ }\right) $ is a prealgebra over $%
\mathcal{F},$ but it is not an algebra over $\mathcal{F}.$ Therefore, the
class of prealgebras is larger than the class of algebras.
\end{example}

\begin{definition}
If $\left( A,\left[ ,\right] _{A}\right) $ and $\left( A^{\prime },\left[ ,%
\right] _{A^{\prime }}\right) $ are (pre)algebras over over $\mathcal{F}$
and $\mathcal{F}^{\prime }$ respectively, then the set 
\begin{equation*}
\left\{ \left( \alpha ,\alpha _{0}\right) \in \mathbf{Mod}\left( A,A^{\prime
}\right) :\alpha \left( \lbrack u,v]_{A}\right) =[\alpha \left( u\right)
,\alpha \left( v\right) ]_{A^{\prime }}\right\} ,
\end{equation*}%
will be called the set of morphisms from $\left( A,\left[ ,\right]
_{A}\right) $ to $\left( A^{\prime },\left[ ,\right] _{A^{\prime }}\right) .$
\end{definition}

We note by $\mathbf{(Pre)Alg}$ the category of (pre)algebras.

\begin{definition}
If $(A,[,]_{A})$ is a (pre)algebra over $\mathcal{F}$ such that $%
[u,u]_{A}=0, $ for any $u\in A$, then we will say that $(A,[,]_{A})$ is an
almost Lie (pre)algebra over $\mathcal{F}$. In addition, if 
\begin{equation*}
\lbrack u,[v,z]_{A}]_{A}+[z,[u,v]_{A}]_{A}+[v,[z,u]_{A}]_{A}=0,
\end{equation*}%
for any $u,v~,z\in A,$ then we say that the triple $(A,[,]_{A})$ is a Lie
(pre)algebra over $\mathcal{F}$.
\end{definition}

We denote by $\mathbf{Lie(Pre)Alg}$ \textit{the category of Lie (pre)algebras%
}.

\begin{example}
$\left( Der(\mathcal{F}),[,]_{Der(\mathcal{F})}^{\circ }\right) $ is a Lie
prealgebra over $\mathcal{F}$.
\end{example}

\section{Generalized (Almost) Lie Algebras/Algebroids}

\ \ \ \ \ \ \ \ 

We begin this section by a new extension of the usual notion of Lie algebra
as the following.

\begin{definition}
A generalized (almost) Lie algebra over unitary and commutative ring $%
\mathcal{F}$ is a triple $(A,[,]_{A},\left( \rho ,\rho _{0}\right) )$
satisfying

$1.$ $(A,[,]_{A})$ is a (almost) Lie prealgebra over $\mathcal{F}$;

$2.$ $\left( \rho ,\rho _{0}\right) $ is a modules morphism from $\left(
A,+,\cdot \right) $ to $\left( Der(\mathcal{F}),+,\cdot \right) $ (called
anchor map) such that $\rho _{0}$ is invertible and satisfies the
compatibility condition 
\begin{equation*}
\lbrack u,f\cdot v]_{A}=f\cdot \lbrack u,v]_{A}+\rho _{0}^{-1}\left( \rho
(u)(\rho _{0}\left( f\right) )\right) \cdot v,
\end{equation*}%
for any $u,v\in A$ and $f\in \mathcal{F}$.

The generalized Lie algebra $(A,[,]_{A},\left( \rho ,Id_{\mathcal{F}}\right)
)$ will be denoted $(A,[,]_{A},\rho ).$
\end{definition}

\begin{remark}
The triple $(A,[,]_{A},\left( \rho ,\rho _{0}\right) )$ is a generalized Lie
algebra over unitary and commutative ring $\mathcal{F}$ if and only if $%
(A,[,]_{A},\left( \rho ,\rho _{0}\right) )$ is a generalized almost Lie
algebra over unitary and commutative ring $\mathcal{F}$ (see \cite{7},\cite%
{8}) such that 
\begin{equation*}
\lbrack u,[v,z]_{A}]_{A}+[z,[u,v]_{A}]_{A}+[v,[z,u]_{A}]_{A}=0,
\end{equation*}%
for any $u,v~,z\in A.$
\end{remark}

\begin{theorem}
If $(A,[,]_{A},\rho )$ is a generalized Lie $\mathcal{F}$-algebra such that $%
A$ is a faithful module, then $\rho $ is a prealgebras morphism from $%
(A,[,]_{A})$ to $(Der(\mathcal{F}),[,]_{Der(\mathcal{F})})$.
\end{theorem}

\begin{proof}
Let $u,v,w\in A$ and $f\in \mathcal{F}$. Since $(A,[,]_{A})$ is a Lie $%
\mathcal{F}$-algebra, we have the Jacobi identity:%
\begin{equation}
\lbrack \lbrack u,v]_{A},fw]_{A}=[u,[v,fw]_{A}]_{A}-[v,[u,fw]_{A}]_{A}.
\label{Jacob}
\end{equation}%
Using the definition of generalized Lie $\mathcal{F}$-algebra, we obtain
\begin{equation*}
\lbrack \lbrack u,v]_{A},fw]_{A}=f[[u,v]_{A},w]_{A}+\rho \lbrack
u,v]_{A}(f)\cdot w.
\end{equation*}%
Similarly, we get
\begin{equation*}
\begin{array}{cl}
\lbrack u,[v,fw]_{A}]_{A} & =[u,f[v,w]_{A}+\rho (v)(f)\cdot w]_{A} \\
& =f[u,[v,w]_{A}]_{A}+\rho (u)(f)\cdot \lbrack v,w]_{A} \\
& \ \ \ +\rho (v)(f)\cdot \lbrack u,w]_{A}+\rho (u)(\rho (v)(f))\cdot w,%
\end{array}%
\end{equation*}%
and%
\begin{equation*}
\begin{array}{cl}
\lbrack v,[u,fw]_{A}]_{A} & =[v,f[u,w]_{A}+\rho (u)(f)\cdot w]_{A} \\
& =f[v,[u,w]_{A}]_{A}+\rho (v)(f)\cdot \lbrack u,w]_{A} \\
& \ \ \ +\rho (u)(f)\cdot \lbrack v,w]_{A}+\rho (v)(\rho (u)(f))\cdot w.%
\end{array}%
\end{equation*}%
Setting three above equations in (\ref{Jacob}) and using Jacobi identity, we
obtain
\begin{equation*}
\begin{array}{cl}
\rho \lbrack u,v]_{A}(f)\cdot w & =\rho (u)(\rho (v)(f))\cdot w-\rho
(v)(\rho (u)(f))\cdot w \\
& =[\rho (u),\rho (v)]_{Der(\mathcal{F})}(f)\cdot w,%
\end{array}%
\end{equation*}%
and consequently
\begin{equation*}
\rho \lbrack u,v]_{A}=[\rho (u),\rho (v)]_{Der(\mathcal{F})}.
\end{equation*}%
Thus, the modules morphism $\rho $ is a prealgebras morphism from $%
(A,[,]_{A})$ to $(Der(\mathcal{F}),[,]_{Der(\mathcal{F})})$.
\end{proof}

\begin{proposition}
If $\left( \left( A,\left[ ,\right] _{A}\right) ,\left( \mathcal{F},\cdot
\right) ,\rho \right) $ is a Lie-Rinehart algebra over the field $\mathbb{K}$%
, then the triple $\left( A,\left[ ,\right] _{A},\rho \right) $ is a
generalized Lie algebra over $\mathcal{F}$ such that its anchor map $\rho $
is a prealgebras morphism from $\left( A,\left[ ,\right] _{A}\right) $ to $%
\left( Der(\mathcal{F}),[,]_{Der(\mathcal{F})}^{\circ }\right) .$
\end{proposition}

\begin{proof}
Because $\left( \left( A,\left[ ,\right] _{A}\right) ,\left( \mathcal{F}%
,\cdot \right) ,\rho \right) $ is a Lie-Rinehart algebra over the field $%
\mathbb{K}$, it results that $\left( Der\left( \mathcal{F}\right) ,\left[ ,%
\right] _{Der\left( \mathcal{F}\right) }^{\circ }\right) $ is a Lie algebra
over $\mathbb{K}$ and $(A,[,]_{A})$ is an Lie prealgebra over $\mathcal{F}.$
Also, the anchor map $\rho $ is a modules morphism from $A$ to $Der\left(
\mathcal{F}\right) $ which satisfies the compatibility condition
\begin{equation*}
\lbrack u,f\cdot v]_{A}=f\cdot \lbrack u,v]_{A}+\rho (u)(f)\cdot v,
\end{equation*}%
for any $u,v\in A$ and $f\in \mathcal{F}.$ After some technical
computations, we obtain that the modules morphism $\rho $ is a morphism of
prealgebras over $\mathcal{F}$, namely%
\begin{equation*}
\rho \left( \left[ u,f\cdot v\right] _{A}\right) =\left[ \rho \left(
u\right) ,f\cdot \rho \left( v\right) \right] _{Der\left( \mathcal{F}\right)
}^{\circ },
\end{equation*}%
for any $u,v\in A$ and $f\in \mathcal{F}$. So, we obtain the conclusion of
the proposition.
\end{proof}

\begin{corollary}
\label{Arcus2} If $\left( A,\left[ ,\right] _{A},\rho \right) $ is a
generalized Lie algebra over $\mathcal{F}$ such that its anchor map $\rho $
is not a prealgebras morphism from $\left( A,\left[ ,\right] _{A}\right) $
to $\left( Der(\mathcal{F}),[,]_{Der(\mathcal{F})}^{\circ }\right) ,$ then
the triple $\left( \left( A,\left[ ,\right] _{A}\right) ,\left( \mathcal{F}%
,\cdot \right) ,\rho \right) $ can not be regarded as a Lie-Rinehart algebra
over a field $\mathbb{K}.$

So, if $\left( A,\left[ ,\right] _{A},\rho \right) $ is a generalized Lie
algebra such that $A$ is not a faithful module over $\mathcal{F},$ then the
triple $\left( \left( A,\left[ ,\right] _{A}\right) ,\left( \mathcal{F}%
,\cdot \right) ,\rho \right) $ can not be regarded as a Lie-Rinehart algebra
over a field $\mathbb{K}.$
\end{corollary}

\begin{example}
If $(L,[,]_{L})$ is a Lie algebra over $\mathcal{F}$, then considering the
null anchor map $0$ from $L$ to $Der(\mathcal{F})$, it results that $%
(L,[,]_{L},0)$ is a generalized Lie algebra over $\mathcal{F}$.
\end{example}

\begin{example}
If $\mathcal{F}$ is an unitary and commutative ring, then considering the
usual Lie bracket $[,]_{Der(\mathcal{F})}^{\circ }$ and the identity anchor
map $Id_{Der(\mathcal{F})},$ we will obtain that the triple 
\begin{equation*}
\left( Der(\mathcal{F}),[,]_{Der(\mathcal{F})}^{\circ },Id_{Der(\mathcal{F}%
)}\right)
\end{equation*}%
is a generalized Lie algebra over $\mathcal{F}$.
\end{example}

In the following we present a generalized Lie algebra which is not
Lie-Rinehart algebra.

\begin{example}
Let $\mathcal{F}$ be an unitary and commutative ring. If $Der(\mathcal{F})$
is a free module with basis $\{\partial _{i}\}_{i\in \overline{1,m}}$ and $%
\rho $ is a modules endomorphism of $Der(\mathcal{F})$, then we define the
application $Der(\mathcal{F})\times Der(\mathcal{F})~^{\underrightarrow{~\ \
\bullet ~\ \ }}~End(\mathcal{F})$ given by the equality 
\begin{equation*}
X\bullet Y=Y^{i}(X\circ \partial _{i})+\rho (X)(Y^{i})\partial _{i},
\end{equation*}%
where $Y=Y^{i}\partial _{i}$ and "$\circ $" is the usual composition
application. Now, we define 
\begin{equation*}
\lbrack X,Y]_{Der(\mathcal{F})}^{\bullet }=X\bullet Y-Y\bullet X,\ \ \forall
X,Y\in Der(\mathcal{F}).
\end{equation*}

Using the Leibniz property of $\partial _{i}$ and $\partial _{j}$ we get 
\begin{equation*}
\lbrack X,Y]_{Der(\mathcal{F})}^{\bullet }(f\cdot g)=[X,Y]_{Der(\mathcal{F}%
)}^{\bullet }(f)\cdot g+f\cdot \lbrack X,Y]_{Der(\mathcal{F})}^{\bullet }(g).
\end{equation*}

Thus, $[X,Y]_{Der(\mathcal{F})}^{\bullet }\in Der(\mathcal{F})$. Moreover,
if $f\in \mathcal{F}$, then we have 
\begin{equation*}
\lbrack X,fY]_{Der(\mathcal{F})}^{\bullet }=f[X,Y]_{Der(\mathcal{F}%
)}^{\bullet }+\rho (X)(f)Y.
\end{equation*}

Easily, we obtain that 
\begin{eqnarray*}
\lbrack f\cdot X,f\cdot X]_{Der(\mathcal{F})}^{\bullet } &=&f\cdot \lbrack
f\cdot X,X]_{Der(\mathcal{F})}^{\bullet }+\rho (f\cdot X)(f)\cdot X \\
&=&f\cdot \left( -[X,f\cdot X]_{Der(\mathcal{F})}^{\bullet }\right) +\left(
f\cdot \rho (X)(f)\right) \cdot X \\
&=&-f^{2}\cdot \lbrack X,X]_{Der(\mathcal{F})}^{\bullet }-f\cdot \left( \rho
(X)(f)\cdot X\right) +\left( f\cdot \rho (X)(f)\right) \cdot X=0.
\end{eqnarray*}

Also, the Jacobi identity holds for this bracket. We remark that if $\rho
\left( \partial _{i}\circ \partial _{j}\right) =$ $\rho \left( \partial
_{i}\right) \circ \rho \left( \partial _{j}\right) ,$ then $\rho $ is a
prealgebras morphism from $\left( Der(\mathcal{F}),[,]_{Der(\mathcal{F}%
)}^{\bullet }\right) $ to $\left( Der(\mathcal{F}),[,]_{Der(\mathcal{F}%
)}^{\circ }\right) .$ As the equality $\rho \left( \partial _{i}\circ
\partial _{j}\right) =$ $\rho \left( \partial _{i}\right) \circ \rho \left(
\partial _{j}\right) $ can not be check, because $\partial _{i}\circ
\partial _{j}\notin Der(\mathcal{F}),$ then $\rho $ can not be a prealgebras
morphism from $\left( Der(\mathcal{F}),[,]_{Der(\mathcal{F})}^{\bullet
}\right) $ to $\left( Der(\mathcal{F}),[,]_{Der(\mathcal{F})}^{\circ
}\right) $. Using Corollary \ref{Arcus2}, it result that the triple $(Der(%
\mathcal{F}),[,]_{Der(\mathcal{F})}^{\bullet },\rho )$ is an example by
generalized Lie algebra and the triple $\left( \left( Der(\mathcal{F}%
),[,]_{Der(\mathcal{F})}^{\bullet }\right) ,\left( \mathcal{F},\cdot \right)
,\rho \right) $ can not be regarded as a Lie-Rinehart algebra over a field $%
\mathbb{K}.$
\end{example}

\begin{definition}
A generalized (almost) Lie algebras morphism from $(A,[,]_{A},\left( \rho
,\rho _{0}\right) )$ over $\mathcal{F}$ to $(A^{\prime },[,]_{A^{\prime
}},\left( \rho ^{\prime },\rho _{0}^{\prime }\right) )$ over $\mathcal{F}%
^{\prime }$ is a couple $\left( \left( a,a_{0}\right) ,\left( b,b_{0}\right)
\right) ,$ where%
\begin{equation*}
\left( a,a_{0}\right) \in \mathbf{PreAlg}\left( \left( A,[,]_{A}\right)
,\left( A^{\prime },[,]_{A^{\prime }}\right) \right)
\end{equation*}%
and 
\begin{equation*}
\left( b,b_{0}\right) \in \mathbf{PreAlg}\left( \left( Der\left( \mathcal{F}%
\right) ,[,]_{Der\left( \mathcal{F}\right) }^{\circ }\right) ,\left(
Der\left( \mathcal{F}^{\prime }\right) ,[,]_{Der\left( \mathcal{F}^{\prime
}\right) }^{\circ }\right) \right)
\end{equation*}%
such that the following diagrams are commutative: 
\begin{equation*}
\begin{array}{rcl}
~\ \ \ ~\ \ \ \ \ A & ^{\underrightarrow{~\ \ a~\ \ }} & A^{\prime } \\ 
\ \rho \downarrow &  & \downarrow \rho ^{\prime } \\ 
~\ \ \ \ \ \ \ Der\left( \mathcal{F}\right) & ^{\underrightarrow{~\ \ b~\ \ }%
} & Der\left( \mathcal{F}^{\prime }\right)%
\end{array}%
\begin{array}{rcl}
~\ \ \ ~\ \ \ \ \ \mathcal{F} & ^{\underrightarrow{~\ \ a_{0}~\ \ }} & 
\mathcal{F}^{\prime } \\ 
\ \rho _{0}\downarrow &  & \downarrow \rho _{0}^{\prime } \\ 
~\ \ \ \ \ \ \ \mathcal{F} & ^{\underrightarrow{~\ \ b_{0}~\ \ }} & \mathcal{%
F}^{\prime }%
\end{array}%
\begin{array}{rcl}
~\ \ \ ~\ \ \ \ \ \mathcal{F} & ^{\underrightarrow{~\ \ a_{0}~\ \ }} & 
\mathcal{F}^{\prime } \\ 
\ \rho \left( u\right) \downarrow &  & \downarrow \rho ^{\prime }\circ
a\left( u\right) \\ 
~\ \ \ \ \ \ \ \mathcal{F} & ^{\underrightarrow{~\ \ b_{0}~\ \ }} & \mathcal{%
F}^{\prime }%
\end{array}%
\end{equation*}%
and%
\begin{equation*}
\left( \rho ^{\prime }\circ a\right) \left( u\right) \circ \left( b\circ
\rho \right) \left( v\right) -\left( \rho ^{\prime }\circ a\right) \left(
v\right) \circ \left( b\circ \rho \right) \left( u\right) \in Der\left( 
\mathcal{F}^{\prime }\right) ,
\end{equation*}%
for any $u,v\in A.$ Every endomorphism $\left( \left( a,a_{0}\right) ,\left(
Id_{Der(\mathcal{F})},Id_{\mathcal{F}}\right) \right) $ will be note simply $%
\left( a,a_{0}\right) .$
\end{definition}

We will note by $\mathbf{gala/gla}$ the category of generalized almost Lie
algebras/generalized Lie algebras.

\begin{remark}
Using the general framework of generalized Lie algebras, a generalized Lie
algebroid can be regarded as a triple $\left( (F,\nu ,N),[,]_{F,h},\left(
\rho ,\eta \right) \right) $ given by the diagrams 
\begin{equation*}
\begin{array}{c}
\begin{array}[b]{ccccc}
(F,[,]_{F,h}) & ^{\underrightarrow{~\ \ \ \rho \ \ \ \ }} & (TM,[,]_{TM}) & 
^{\underrightarrow{~\ \ \ Th\ \ \ \ }} & (TN,[,]_{TN}) \\ 
~\downarrow \nu &  & ~\ \ \downarrow \tau _{M} &  & ~\ \ \ \tau
_{N}\downarrow \ \ \ \ \ \ \  \\ 
N & ^{\underrightarrow{~\ \ \ \eta ~\ \ }} & M & ^{\underrightarrow{~\ \ \
h~\ \ }} & N%
\end{array}%
\end{array}%
,
\end{equation*}%
where $h$ and $\eta $ are arbitrary diffeomorphisms, $(\rho ,\eta )$ is a
vector bundles morphism from $(F,\nu ,N)$ to $(TM,\tau _{M},M)$ and $%
[,]_{F,h}$ is an operation on $\Gamma (F,\nu ,N)$ such that the triple 
\begin{equation*}
\left( \Gamma \left( F,\nu ,N\right) ,\left[ ,\right] _{F,h},\Gamma (Th\circ
\rho ,h\circ \eta )\right)
\end{equation*}%
is a generalized Lie algebra over $\mathcal{F}(N),$ where the anchor map $%
\Gamma (Th\circ \rho ,h\circ \eta )$ is given by the equality%
\begin{equation*}
\left( \Gamma (Th\circ \rho ,h\circ \eta )(u^{\alpha }t_{\alpha })\right)
\left( f\right) =u^{\alpha }\rho _{\alpha }^{i}\cdot \left( \frac{\partial
\left( f\circ h\right) }{\partial x^{i}}\circ h^{-1}\right) ,
\end{equation*}%
for any $u^{\alpha }t_{\alpha }\in \Gamma (F,\nu ,N)$ and $f\in \mathcal{F}%
\left( N\right) $.
\end{remark}

\begin{proposition}
A generalized Lie algebroid is a distinguished example by Lie algebroid.
\end{proposition}

\begin{proof}
Let $\left( (F,\nu ,N),[,]_{F,h},\left( \rho ,\eta \right) \right) $ be a
generalized Lie algebroid. Because
\begin{equation*}
\frac{\partial \left( f\circ h\right) }{\partial x^{i}}\circ h^{-1}=\frac{%
\partial h^{\tilde{\imath}}}{\partial x^{i}}\circ h^{-1}\cdot \frac{\partial
f}{\partial \varkappa ^{\tilde{\imath}}},
\end{equation*}%
then
\begin{eqnarray*}
\left( \Gamma (Th\circ \rho ,h\circ \eta )(u^{\alpha }t_{\alpha })\right)
\left( f\right) &=&\left( u^{\alpha }\cdot \rho _{\alpha }^{i}\cdot \frac{%
\partial h^{\tilde{\imath}}}{\partial x^{i}}\circ h^{-1}\right) \frac{%
\partial f}{\partial \varkappa ^{\tilde{\imath}}} \\
\overset{put}{=}\left( u^{\alpha }\theta _{\alpha }^{\tilde{\imath}}\frac{%
\partial }{\partial \varkappa ^{\tilde{\imath}}}\right) \left( f\right)
&=&\Gamma \left( \theta ,Id_{N}\right)(u^{\alpha }t_{\alpha }) \left( f\right) ,
\end{eqnarray*}%
where $\left( \theta ,Id_{N}\right) $ is a vector bundles morphism from $%
(F,\nu ,N)$ to $(TN,\tau _{N},N).$ Therefore, we obtain that the anchor map $%
\Gamma (Th\circ \rho ,h\circ \eta )$ is a modules morphism. As $\Gamma
\left( F,\nu ,N\right) $ is a faithful module, then the anchor map of the
generalized Lie algebra
\begin{equation*}
\left( \Gamma \left( F,\nu ,N\right) ,\left[ ,\right] _{F,h},\Gamma \left(
\theta ,Id_{N}\right) \right),
\end{equation*}%
is a prealgebra morphism. Therefore, the triple
\begin{equation*}
\left( \left( \Gamma \left( F,\nu ,N\right) ,\left[ ,\right] _{F,h}\right)
,\left( \mathcal{F}(N),\cdot \right) ,\Gamma \left( \theta ,Id_{N}\right),
\right)
\end{equation*}%
is a Lie-Rinehart algebra over $\mathbb{R}$ and so, the triple
\begin{equation*}
\left( \left( F,\nu ,N\right) ,\left[ ,\right] _{F,h},\left( \theta
,Id_{N}\right) \right),
\end{equation*}%
is a Lie algebroid.
\end{proof}

\section{Why the Proof of Theorem 3.1 in \protect\cite{19} Does Not Work?}

\ \ \ \ \ \ \ This section is devoted to detail that Theorem 3.1 in \cite{19}
is based on a completely false assumption and so is not valid. Indeed,
component-wise composition of vector bundle morphism has no sense as a new
vector bundle morphism; and this is their false assumption.

In the theory of vector bundles, it is very famous that if $\left( \varphi
,\varphi _{0}\right) $ and $\left( \psi ,\psi _{0}\right) $ are two vector
bundles morphisms given by the commutative diagrams%
\begin{equation*}
\begin{array}{ccccc}
~\ \ F & ^{\underrightarrow{~\ \ \varphi ~\ \ }} & ~\ \ E & ^{%
\underrightarrow{~\ \ \psi ~\ \ }} & ~\ \ G \\ 
\nu \downarrow & ~\  & \pi \downarrow & ~\  & \tau \downarrow \\ 
~\ \ N & ^{\underrightarrow{~\ \ \varphi _{0}~\ \ }} & ~\ \ M & ^{%
\underrightarrow{~\ \ \psi _{0}~\ \ }} & ~\ \ P \\ 
~\ \left( t_{\alpha }\right) &  & ~\ \left( s_{a}\right) &  & ~\ \left( \xi
_{i}\right)%
\end{array}%
,
\end{equation*}%
such that $\varphi _{0}\in Diff\left( N,M\right) $ and $\psi _{0}\in
Diff\left( M,P\right) ,$ then we can discuss only about the composition
vector bundles morphism $\left( \psi ,\psi _{0}\right) \circ \left( \varphi
,\varphi _{0}\right) $ such that 
\begin{eqnarray*}
\Gamma \left( \left( \psi ,\psi _{0}\right) \circ \left( \varphi ,\varphi
_{0}\right) \right) \left( z^{\alpha }\cdot t_{\alpha }\right) &=&\Gamma
\left( \psi ,\psi _{0}\right) \circ \Gamma \left( \varphi ,\varphi
_{0}\right) \left( z^{\alpha }\cdot t_{\alpha }\right) \\
&=&\Gamma \left( \psi ,\psi _{0}\right) \left( \left( z^{\alpha }\cdot
\varphi _{\alpha }^{a}\right) \circ \varphi _{0}^{-1}\cdot s_{a}\right) \\
&=&\left( \left( \left( z^{a}\cdot \varphi _{\alpha }^{a}\right) \circ
\varphi _{0}^{-1}\right) \cdot \psi _{a}^{i}\right) \circ \psi
_{0}^{-1}\cdot \xi _{i}.
\end{eqnarray*}%
Indeed, it is trivial that the pair $\left( \psi \circ \varphi ,\psi
_{0}\circ \varphi _{0}\right) $ can not be regarded as a vector bundles
morphism. %\end{remark}

So, we can establish the following for generalized Lie algebroids as special
kinds of data structures dealing with the vector bundles.

\begin{corollary}
If $\left( (F,\nu ,N),[,]_{F,h},\left( \rho ,\eta \right) \right) $ is a
generalized Lie algebroid, then we can not discuss about the vector bundles
morphism $\Phi \overset{put}{=}Th\circ \rho $ covering $\phi \overset{put}{=}%
h\circ \eta .$
\end{corollary}

\begin{remark}
If $\left( \varphi ,\varphi _{0}\right) $ and $\left( \psi ,\varphi
_{0}^{-1}\right) $ are two vector bundles morphisms given by the commutative
diagrams:%
\begin{equation*}
\begin{array}{ccccc}
~\ \ F & ^{\underrightarrow{~\ \ \varphi ~\ \ }} & ~\ \ E & ^{%
\underrightarrow{~\ \ \psi ~\ \ }} & ~\ \ G \\ 
\nu \downarrow & ~\  & \pi \downarrow & ~\  & \tau \downarrow \\ 
~\ \ N & ^{\underrightarrow{~\ \ \varphi _{0}~\ \ }} & ~\ \ M & ^{%
\underrightarrow{~\ \ \varphi _{0}^{-1}~\ \ }} & ~\ \ N \\ 
~\ \left( t_{\alpha }\right) &  & ~\ \left( s_{a}\right) &  & ~\ \left( \xi
_{i}\right)%
\end{array}%
,
\end{equation*}%
then the pair $\left( \psi \circ \varphi ,Id_{N}\right) \overset{put}{=}%
\left( \psi \circ \varphi ,\varphi _{0}^{-1}\circ \varphi _{0}\right) $ can
not be regarded as a vector bundles morphism from $\left( F,\nu ,N\right) $
to $\left( G,\tau ,N\right) .$ We can discuss only about the composition
vector bundles morphism $\left( \psi ,\varphi _{0}^{-1}\right) \circ \left(
\varphi ,\varphi _{0}\right) .$
\end{remark}

Finally, by notations used in \cite{19} and gathering above discussions, we
derive the following statement.

\begin{corollary}
If $\left( (F,\nu ,N),[,]_{F,h},\left( \rho ,\eta \right) \right) $ is a
generalized Lie algebroid, then we can not discuss about the vector bundles
morphism $\left( T\phi \right) ^{-1}\circ \Phi $ covering $Id_{N}=\phi
^{-1}\circ \phi .$

So, the affirmation \textit{"... then }$\left[ ,\right] _{\Phi }$\textit{\
is a Lie algebroid bracket on }$F$\textit{\ with respect to a
\textquotedblleft traditional\textquotedblright\ anchor map }$\Psi
:F\longrightarrow TN$\textit{\ covering the identity, }$\Psi =\left( T\phi
\right) ^{-1}\circ \Phi $\textit{."}\ is false.

Therefore, the proof of Theorem 3.1 from \cite{19} is based on a
misconception and is not legal. It breaks Theorem 1.3 down as the only
statement in \cite{19}.
\end{corollary}

\section{Optimal control problem}

\ \ \ \ \ \ \ \ \ 

Here, we detalied the optimal control problem presented in the introduction.
The result shall appoint the importance of generalized Lie algebroids in
real occasions.

If, for any $x\in \Sigma ,$ we consider $(\tilde{x}^{1},\tilde{x}^{2},\tilde{%
x}^{3})=\varphi _{\Sigma }\circ s_{O}(x),$ then the equations system (\ref%
{OP}) is equivalent to the following%
\begin{equation*}
\begin{array}{c}
\frac{d\tilde{x}^{1}}{dt}=-\tilde{x}^{2}y^{2}-y^{3}, \\ 
\frac{d\tilde{x}^{2}}{dt}=-\tilde{x}^{1}y^{1}-\tilde{x}^{2}y^{2}-y^{3}, \\ 
\frac{d\tilde{x}^{3}}{dt}=-y^{1}.%
\end{array}%
\end{equation*}

As 
\begin{equation*}
\left( 
\begin{array}{c}
\frac{d\tilde{x}^{1}}{dt} \\ 
\frac{d\tilde{x}^{2}}{dt} \\ 
\frac{d\tilde{x}^{3}}{dt}%
\end{array}%
\right) =\left( 
\begin{array}{ccc}
0 & -\tilde{x}^{2} & -1 \\ 
-\tilde{x}^{1} & -\tilde{x}^{2} & -1 \\ 
-1 & 0 & 0%
\end{array}%
\right) \cdot \left( 
\begin{array}{c}
y^{1} \\ 
y^{2} \\ 
y^{3}%
\end{array}%
\right) ,
\end{equation*}%
and 
\begin{equation*}
\left( 
\begin{array}{ccc}
0 & -\tilde{x}^{2} & -1 \\ 
-\tilde{x}^{1} & -\tilde{x}^{2} & -1 \\ 
-1 & 0 & 0%
\end{array}%
\right) =\left( 
\begin{array}{cc}
\tilde{x}^{1} & -1 \\ 
0 & -1 \\ 
-1 & 0%
\end{array}%
\right) \cdot \left( 
\begin{array}{ccc}
1 & 0 & 0 \\ 
\tilde{x}^{1} & \tilde{x}^{2} & 1%
\end{array}%
\right) ,
\end{equation*}%
it results that we can consider the vector subbundle $(F,\tau _{\Sigma
},\Sigma )$ of the vector bundle $(T\Sigma ,\tau _{\Sigma },\Sigma )$ such
that $\Gamma (F,\tau _{\Sigma },\Sigma )=Span(t_{1},t_{2}),$ where 
\begin{equation*}
t_{1}=\frac{\partial }{\partial \tilde{x}^{1}},\ \ \ t_{2}=\tilde{x}^{1}%
\frac{\partial }{\partial \tilde{x}^{1}}+\tilde{x}^{2}\frac{\partial }{%
\partial \tilde{x}^{2}}+\frac{\partial }{\partial \tilde{x}^{3}}.
\end{equation*}

Denoting by $(\rho ,Id_{\Sigma })$ the vector bundles morphism from $(F,\tau
_{\Sigma },\Sigma )$ to $(T\Sigma ,\tau _{\Sigma },\Sigma )$ given by 
\begin{equation*}
\Gamma (\rho ,Id_{\Sigma })\left( 
\begin{array}{c}
t_{1} \\ 
t_{2}%
\end{array}%
\right) =\left( 
\begin{array}{ccc}
1 & 0 & 0 \\ 
\tilde{x}^{1} & \tilde{x}^{2} & 1%
\end{array}%
\right) \left( 
\begin{array}{c}
\frac{\partial }{\partial \tilde{x}^{1}} \\ 
\frac{\partial }{\partial \tilde{x}^{2}} \\ 
\frac{\partial }{\partial \tilde{x}^{3}}%
\end{array}%
\right) ,
\end{equation*}%
we obtain the Lie algebroid $((F,\tau _{\Sigma },\Sigma ),\left[ ,\right]
_{T\Sigma },(\rho ,Id_{\Sigma }))$. Now, we denote by $(g,s_{O})$ the vector
bundles morphism from $(T\Sigma ,\tau _{\Sigma },\Sigma )$ to $(F,\tau
_{\Sigma },\Sigma )$ given by 
\begin{equation*}
\Gamma (g,s_{O})\left( 
\begin{array}{c}
\frac{\partial }{\partial \tilde{x}^{1}} \\ 
\frac{\partial }{\partial \tilde{x}^{2}} \\ 
\frac{\partial }{\partial \tilde{x}^{3}}%
\end{array}%
\right) =\left( 
\begin{array}{cc}
\tilde{x}^{1} & -1 \\ 
0 & -1 \\ 
-1 & 0%
\end{array}%
\right) \left( 
\begin{array}{c}
t_{1} \\ 
t_{2}%
\end{array}%
\right) ,
\end{equation*}%
and we consider the vector bundles morphism $(Ts_{O},s_{O})$ from $(T\Sigma
,\tau _{\Sigma },\Sigma )$ to $(T\Sigma ,\tau _{\Sigma },\Sigma )$ given by 
\begin{equation*}
\begin{array}{c}
\Gamma (Ts_{O},s_{O})\left( 
\begin{array}{c}
\frac{\partial }{\partial x^{1}} \\ 
\frac{\partial }{\partial x^{2}} \\ 
\frac{\partial }{\partial x^{3}}%
\end{array}%
\right) =\left( 
\begin{array}{ccc}
-1 & 0 & 0 \\ 
0 & -1 & 0 \\ 
0 & 0 & -1%
\end{array}%
\right) \left( 
\begin{array}{c}
\frac{\partial }{\partial \tilde{x}^{1}} \\ 
\frac{\partial }{\partial \tilde{x}^{2}} \\ 
\frac{\partial }{\partial \tilde{x}^{3}}%
\end{array}%
\right) ,%
\end{array}%
\end{equation*}%
where $(\frac{\partial }{\partial \tilde{x}^{1}},\frac{\partial }{\partial 
\tilde{x}^{2}},\frac{\partial }{\partial \tilde{x}^{3}})$ is the natural
base. Since 
\begin{equation*}
\left( 
\begin{array}{cc}
\tilde{x}^{1} & -1 \\ 
0 & -1 \\ 
-1 & 0%
\end{array}%
\right) =\left( 
\begin{array}{ccc}
-1 & 0 & 0 \\ 
0 & -1 & 0 \\ 
0 & 0 & -1%
\end{array}%
\right) \cdot \left( 
\begin{array}{cc}
-\tilde{x}^{1} & 1 \\ 
0 & 1 \\ 
1 & 0%
\end{array}%
\right) ,
\end{equation*}%
then we deduce 
\begin{equation*}
\Gamma (g,s_{O})=\Gamma (R,Id_{\Sigma })\circ \Gamma (Ts_{O},s_{O}),
\end{equation*}%
where $(R,Id_{\Sigma })$ is a vector bundles morphism from $(T\Sigma ,\tau
_{\Sigma },\Sigma )$ to $(F,\tau _{\Sigma },\Sigma )$ given by 
\begin{equation*}
\Gamma (R,Id_{\Sigma })\left( 
\begin{array}{c}
\frac{\partial }{\partial \tilde{x}^{1}} \\ 
\frac{\partial }{\partial \tilde{x}^{2}} \\ 
\frac{\partial }{\partial \tilde{x}^{3}}%
\end{array}%
\right) =\left( 
\begin{array}{cc}
-\tilde{x}^{1} & 1 \\ 
0 & 1 \\ 
1 & 0%
\end{array}%
\right) \left( 
\begin{array}{c}
t_{1} \\ 
t_{2}%
\end{array}%
\right) .
\end{equation*}

If we denote by $\mathcal{R}$ the matrix 
\begin{equation*}
\left( 
\begin{array}{cc}
-\tilde{x}^{1} & 1 \\ 
0 & 1 \\ 
1 & 0%
\end{array}%
\right) ,
\end{equation*}%
then 
\begin{equation*}
\mathcal{R}^{t}=\left( 
\begin{array}{ccc}
-\tilde{x}^{1} & 0 & 1 \\ 
1 & 1 & 0%
\end{array}%
\right) ,~\mathcal{R}^{t}\circ \mathcal{R}=\left( 
\begin{array}{cc}
1+(\tilde{x}^{1})^{2} & -\tilde{x}_{1} \\ 
-x_{1} & 2%
\end{array}%
\right) \text{and }\det (\mathcal{R}^{t}\circ \mathcal{R})=2+(\tilde{x}%
^{1})^{2}\neq 0.
\end{equation*}

Easily, we obtain 
\begin{equation*}
(\mathcal{R}^{t}\circ \mathcal{R})^{-1}=\frac{1}{2+(\tilde{x}^{1})^{2}}%
\left( 
\begin{array}{cc}
2 & \tilde{x}_{1} \\ 
\tilde{x}_{1} & 1+(\tilde{x}^{1})^{2}%
\end{array}%
\right) ,
\end{equation*}%
and so 
\begin{equation*}
\mathcal{R}_{left}^{-1}\overset{put}{=}(\mathcal{R}^{t}\circ \mathcal{R}%
)^{-1}\circ \mathcal{R}^{t}=\frac{1}{2+(\tilde{x}^{1})^{2}}\left( 
\begin{array}{ccc}
-\tilde{x}^{1} & \tilde{x}^{1} & 2 \\ 
1 & 1+(\tilde{x}^{1})^{2} & \tilde{x}^{1}%
\end{array}%
\right) ,
\end{equation*}%
is the left inverse of the matrix $\mathcal{R}$, because $\mathcal{R}%
_{left}^{-1}\cdot \mathcal{R}=I_{2}$. Thus we obtain the left inverse vector
bundle morphism $(R_{left}^{-1},Id_{\Sigma })$ from $(F,\tau _{\Sigma
},\Sigma )$ to $(T\Sigma ,\tau _{\Sigma },\Sigma )$\ given by 
\begin{equation*}
\Gamma (R_{left}^{-1},Id_{\Sigma })\left( 
\begin{array}{c}
t_{1} \\ 
t_{2}%
\end{array}%
\right) =\frac{1}{2+(\tilde{x}^{1})^{2}}\left( 
\begin{array}{ccc}
-\tilde{x}^{1} & \tilde{x}^{1} & 2 \\ 
1 & 1+(\tilde{x}^{1})^{2} & \tilde{x}^{1}%
\end{array}%
\right) \cdot \left( 
\begin{array}{c}
\frac{\partial }{\partial \tilde{x}^{1}} \\ 
\frac{\partial }{\partial \tilde{x}^{2}} \\ 
\frac{\partial }{\partial \tilde{x}^{3}}%
\end{array}%
\right) .
\end{equation*}

If $\tilde{g}=Ts_{O}^{-1}\circ R_{left}^{-1},$ then $(\tilde{g},s_{O}^{-1})$
is a vector bundles morphism from $(F,\tau _{\Sigma },\Sigma )$ to $(T\Sigma
,\tau _{\Sigma },\Sigma )$ and 
\begin{equation*}
\Gamma (\tilde{g},s_{O}^{-1})\left( 
\begin{array}{c}
t_{1} \\ 
t_{2}%
\end{array}%
\right) =\left( 
\begin{array}{ccc}
\frac{x^{1}}{2+(x^{1})^{2}} & \frac{-x^{1}}{2+(x^{1})^{2}} & \frac{-2}{%
2+(x^{1})^{2}} \\ 
\frac{-1}{2+(x^{1})^{2}} & \frac{-1-(x^{1})^{2}}{2+(x^{1})^{2}} & \frac{%
-x^{1}}{2+(x^{1})^{2}}%
\end{array}%
\right) \left( 
\begin{array}{c}
\frac{\partial }{\partial x^{1}} \\ 
\frac{\partial }{\partial x^{2}} \\ 
\frac{\partial }{\partial x^{3}}%
\end{array}%
\right) .
\end{equation*}

As $\tilde{g}_{\alpha }^{a}\cdot g_{b}^{\alpha }=\delta _{b}^{a},$ it
results that the vector bundles morphism $(g,s_{O})$ is left invertible and $%
(\tilde{g},s_{O}^{-1})$ is its left inverse. So, we pass the diagram 
\begin{equation*}
\begin{array}{ccccccccc}
&  & T\Sigma  & ^{\underrightarrow{~\ \ g~\ \ }} & (F,[,]_{F,s_{O}}) & ^{%
\underrightarrow{~\ \ \rho ~\ \ }} & T\Sigma  & ^{\underrightarrow{~\ \
Ts_{O}~\ \ }} & T\Sigma  \\ 
& \dot{c}\nearrow  & ~\ \downarrow \tau _{\Sigma } &  & \downarrow \tau
_{\Sigma } &  & ~\ \ \downarrow \tau _{\Sigma } &  & ~\ \ \downarrow \tau
_{\Sigma } \\ 
I & ^{\underrightarrow{~\ \ c~\ \ }} & \Sigma  & ^{\underrightarrow{~\ \
s_{O}~\ \ }} & \Sigma  & ^{\underrightarrow{~\ \ Id_{\Sigma }~\ \ }} & 
\Sigma  & ^{\underrightarrow{~\ \ s_{O}~\ \ }} & \Sigma 
\end{array}%
,
\end{equation*}%
where the vector bundle $\left( T\Sigma ,\tau _{\Sigma },\Sigma \right) $ is
anchored by the generalized Lie algebroid $\left( (F,\nu
,N),[,]_{F,s_{O}},\left( \rho ,Id_{\Sigma }\right) \right) $ with the help
of a left invertible vector bundles morphism $\left( g,s_{O}\right) .$

\bigskip 

In the end of this paper, we ask:

\begin{quotation}
\textbf{- }Can we develop a Lagrangian formalism directly on a vector bundle 
$\left( E,\pi ,M\right) $ anchored by a (generalized) Lie algebroid $\left(
(F,\nu ,N),[,]_{F,h},\left( \rho ,\eta \right) \right) $ with the help of a
left invertible vector bundles morphism $\left( g,h\right) $ similar to
Klein's formalism for ordinary Lagrangian Mechanics?
\end{quotation}

We suppose that it is possible, but it is necessary to extend the notion of
pullback vector bundle and, using it, we can obtain a new version of the Lie
algebroid generalized tangent bundle or of the prolongation Lie algebroid.
Also, we suppose that, in particular situations, this space was used in all
our papers including theory of connections \cite{2}, mechanics and optimal
control \cite{3}, Kaluza-Klein $G$-spaces \cite{4}, Weil's theory \cite{5},
vertical and complete lifts \cite{6}.

\textbf{Acknowledgment}

\addcontentsline{toc}{section}{Acknowledgment}

We like thank to R\u{a}dine\c{s}ti-Gorj Cultural Scientific Society for
financial support. We are deeply indebted to Professor Janusz Grabowski for
his essential contribution to the discussions about generalized (almost) Lie
algebras/algebroids. %%

%
%===========================
%
{\small \ }

\noindent
\begin{verbatim}
c_arcus@radinesti.ro, e-peyghan@araku.ac.ir, esasharahi@gmail.com
Secondary School "Cornelius Radu", Radinesti Village, 217196, Gorj
County, Romania,
Department of Mathematics, Faculty of Science, Arak
University, Arak 38156-8-8349, Iran.
\end{verbatim}

%==================

\end{document}